\documentclass{article}%
\usepackage{graphicx}
\usepackage{amsmath,amssymb,amsfonts}%
\usepackage{theorem}%
\usepackage{color}%
\usepackage{hyperref}
\usepackage[font={small, it}]{caption}
\usepackage{MnSymbol}
\usepackage{caption}
\usepackage{subcaption}
\usepackage{tikz}

\theoremstyle{change}%
\newtheorem{definition}{Definition:}[section]%
\newtheorem{proposition}[definition]{Proposition:}%
\newtheorem{theorem}[definition]{Theorem:}%
\newtheorem{lemma}[definition]{Lemma:}%
\newtheorem{corollary}[definition]{Corollary:}%
{\theorembodyfont{\rmfamily}\newtheorem{remark}[definition]{Remark:}}%
{\theorembodyfont{\rmfamily}}%

\newenvironment{proof}
{{\bf Proof:}}
{\qquad \hspace*{\fill} $\Box$}%

\newcommand{\fg}{\mathfrak{g}}%
\newcommand{\fz}{\mathfrak{z}}%

\newcommand{\Ad}{\operatorname{Ad}}%
\newcommand{\ad}{\operatorname{ad}}%
\newcommand{\tr}{\operatorname{tr}}%
\newcommand{\id}{\operatorname{id}}
\newcommand{\inner}{\operatorname{int}}%
\newcommand{\fix}{\operatorname{fix}}%
\newcommand{\rme}{\mathrm{e}}%
\newcommand{\EC}{\mathcal{E}}%
\newcommand{\NC}{\mathcal{N}}%
\newcommand{\OC}{\mathcal{O}}%
\newcommand{\RC}{\mathcal{R}}%
\newcommand{\UC}{\mathcal{U}}%
\newcommand{\XC}{\mathcal{X}}%
\newcommand{\DC}{\mathcal{D}}%
\newcommand{\HC}{\mathcal{H}}%

\newcommand{\N}{\mathbb{N}}%
\newcommand{\R}{\mathbb{R}}%

\begin{document}

	\title{Dynamics of linear control systems and stabilization}%
	\author{V{\'\i}ctor Ayala \\Instituto
de Alta Investigaci{\'o}n\\Universidad de Tarapac{\'a}, Arica, Chile 
\and
Adriano Da Silva\footnote{adasilva@academicos.uta.cl }\\Departamento de Matem{\'a}tica,\\Universidad
de Tarapac{\'a}, Arica, Chile}

	\date{\today }
	\maketitle

\begin{abstract}
In this paper, we study linear control systems with positive bounded orbits. We show that the existence of positive bounded orbits imposes strong algebraic and topological constraints on the state space. In fact, a linear control system has bounded positive orbits if and only if it can be decomposed as the product of the stable and central subgroups of the drift (see Section 2.2 below), with the central subgroup being compact. In particular, systems with bounded positive orbits admit a compact control set, and if the system is controllable, the entire state space is a compact group. As a byproduct, we obtain a complete characterization of the internal and BIBO stability of linear control systems.
\end{abstract}

	

\maketitle

\section{Introduction}

The study of dynamical systems on a Lie group $G$, and in particular the class of linear control systems on $G$, plays a relevant role in geometric control theory. These systems evolve on manifolds 
endowed with a rich algebraic structure, and their analysis requires a careful 
combination of differential geometry, Lie theory, and dynamical systems methods. 
A defining feature of these dynamics is that a linear vector field serves as the drift of a linear control system whose flow forms a one-parameter group of automorphisms of the underlying Lie group. This structure guarantees that the system’s evolution is consistent with the intrinsic symmetries of the state space. Although such control systems are primarily of theoretical interest, it has been shown \cite[Theorem 4]{JPh1} that, under suitable conditions, general nonlinear control-affine systems can be globally equivalent to linear control systems on Lie groups or homogeneous spaces, which highlights their applied relevance as well.
In the present work, we study linear control systems with bounded orbits. We begin by analyzing bounded orbits associated with the trivial control, or equivalently, with the flow of the drift. In particular, we focus first on orbits starting at points in the central subgroup. Using the Jordan decomposition of linear vector fields on Lie groups, we show that an orbit starting at a point in the central subgroup is bounded if and only if that point belongs to the set of fixed points of the nilpotent part. As a consequence, all points in the central subgroup have bounded orbits if and only if the restriction of the flow to this subgroup is elliptic.

This description allows us to establish that the stable and unstable manifolds of the flow at the identity element coincide precisely with the stable and unstable subgroups of the flow. In particular, any orbit starting in the stable manifold is bounded for positive time, while those starting in the unstable manifold are bounded for negative time. Subsequently, we extend these results to bounded orbits associated with arbitrary bounded controls. Our main theorem shows that the orbits in positive time of a linear control system are bounded for any control if and only if the entire positive orbit from the identity is bounded, which holds precisely when the group decomposes as the product of the stable and central subgroups, with the central subgroup compact. In particular, systems with bounded positive orbits admit a compact, positively invariant control set, implying that, in this setting, controllability of the system entails compactness of the group.
As an application of the previous dynamical analysis, we consider the stability of linear control systems. Stability plays a central role in control theory due to its wide practical relevance. Examples include gyroscopic stabilization \cite{BlochMarsden1990}, quantum control systems \cite{D'Alessandro2008}, and observer or output design via homomorphisms. The bounded-input, bounded-output (BIBO) criterion ensures that bounded controls produce bounded outputs, providing a practical test of robustness in applications such as sensor networks and autonomous vehicles \cite{Khalil2002, ThrunBurgardFox2005}.

Despite its importance, stability in the context of linear control systems on Lie groups has not yet been fully addressed. The notions of internal stability and bounded-input, bounded-output (BIBO) stability, well established in the Euclidean setting (see for instance \cite{Bacc, Sab}), remain to be extended and characterized in this more general framework. In this work, we provide such a generalization and, by leveraging the results obtained for bounded orbits, give a complete characterization of these stability notions. This approach is inspired by the Euclidean case, where internal and BIBO stability are closely related to the dynamical properties of the system’s drift \cite[Chapter 6]{Bacc}.

The paper is organized as follows. Section 2 reviews linear dynamics on vector spaces and linear vector fields on Lie groups, establishing several structural lemmas. In particular, we introduce the Jordan decomposition of a linear vector field and its associated flow. Section 3 employs this decomposition to study bounded orbits of linear vector fields and to characterize the stable, unstable, and central subgroups from a dynamical perspective. Section 4 examines bounded orbits for linear control systems, showing that the existence of such orbits imposes strong topological constraints on the underlying group. Finally, Section 5 introduces the notions of internal and BIBO stability and applies the preceding dynamical results to characterize these stability concepts in terms of the stable, unstable, and central subgroups.

\subsection*{Notations}

If $H \subset G$ is a subgroup, we denote by $H_1$ the connected component of $H$ containing the identity element $e \in G$. 
If $H = \{e\}$, we say that $H$ is a \emph{trivial} subgroup of $G$. For each $g \in G$, we denote by $L_g$ and $R_g$ the left and right translations by $g$, respectively. 
The conjugation by $g$ is the map $C_g := L_g \circ R_{g^{-1}}$. The center of $G$, denoted by $Z(G)$, is the set of all elements $g \in G$ such that $C_g = \id_G$. Its Lie algebra $\fz(\fg)$ is the set of all elements in $\fg$ that have trivial brackets with any other element of $\fg$. If $f : G \to H$ is a differentiable map between Lie groups, the differential of $f$ at $x \in G$ is denoted by $(df)_x$. If $\varphi:\R\rightarrow G\rightarrow G$ is a flow on $G$. The set of \emph{fixed points} and the set of \emph{recurrent points} of $\varphi$ are defined, respectively, by
$$\fix(\varphi):=\{g\in G; \varphi_t(g)=g, \forall t\in\R\}\hspace{.5cm}\mbox{ and } \hspace{.5cm}\RC(\varphi):=\{g\in G; \exists t_k\rightarrow+\infty; \varphi_{t_k}(g)\rightarrow g\}.$$
We say that a subset $A\subset G$ is $\varphi$-positively (negatively), invariant, if $$\varphi_t(A)\subset A,\hspace{.5cm} \forall t>0 \hspace{.5cm}(\varphi_t(A)\subset A, \hspace{.5cm}\forall t<0).$$
The subset $A$ is said to be $\varphi$-invariant if it is $\varphi$-positively and negatively invariant

\section{Preliminaries}

In this section we present some basic results concerning basic properties of linear vector fields that will be needed for the rest of the paper.

\subsection{Matrices dynamics on vector spaces}

Let $V$ be a finite-dimensional real vector space and $A:V\to V$ a linear map. 
Recall that $A$ is said to be \emph{semisimple} if its complexification $A_{\mathbb{C}}$ is diagonalizable. 
The map $A$ is \emph{nilpotent} if $A^n=0$ for some $n\in\mathbb{N}$, and \emph{elliptic} (respectively, \emph{hyperbolic}) if it is semisimple and all its eigenvalues are purely imaginary (respectively, real).  

The \emph{Jordan decomposition} of $A$ is given by
\[
A = A_{\EC} + A_{\HC} + A_{\NC}, 
\qquad 
\text{where } A, A_{\EC}, A_{\HC}, \text{ and } A_{\NC} \text{ commute},
\]
with $A_{\EC}$ elliptic, $A_{\HC}$ hyperbolic, and $A_{\NC}$ nilpotent.

For each eigenvalue $\lambda$ of $A_{\HC}$, consider 
\[
V_{\lambda} := \{ v \in V \; ; \; A_{\HC} v = \lambda v \}
\]
be the corresponding eigenspace. Define the $A$, $A_{\EC}$, $A_{\HC}$, and $A_{\NC}$-invariant subspaces of $V$ by
\[
V^+ = \bigoplus_{\lambda>0} V_{\lambda}, 
\qquad
V^0 = \ker A_{\HC},
\qquad
V^- = \bigoplus_{\lambda<0} V_{\lambda}.
\]
Then $V = V^+ \oplus V^0 \oplus V^-$, and for any norm on $V$ there exist constants $\lambda>0$, $c \geq 1$ such that
\begin{equation}
\label{expanding}
\|\mathrm{e}^{tA}|_{V^+}\| \geq c\mathrm{e}^{t\lambda},
\qquad
\|\mathrm{e}^{tA}|_{V^-}\| \leq c^{-1}\mathrm{e}^{-t\lambda},
\qquad
\text{for all } t \geq 0.
\end{equation}
Moreover, there exists an inner product $\langle \cdot, \cdot \rangle$ on $V$ such that $\mathrm{e}^{tA_{\EC}}$ is an isometry for every $t\in\mathbb{R}$.

The following lemma, whose proof can be found in \cite[Lemma~2.1]{ASVAPJ}, will be considered essential in the study of linear vector fields.

\begin{lemma}
\label{beta}
Let $\beta:\mathbb{R} \to V$ be a continuous curve satisfying
\[
\beta_{t+s} = \beta_t + \mathrm{e}^{tA}\beta_s, \qquad \forall\, t,s\in\mathbb{R}.
\]
If $A$ has no elliptic part and the sequence $(\beta_{t_k})_{k\in\mathbb{N}}$ is bounded for some $t_k \to \pm\infty$, 
then $\beta_t \in V^{\mp}$ for all $t\in\mathbb{R}$.  
In particular, $(\beta_t)_{t\ge0}$ is bounded.
\end{lemma}

\subsection{Linear vector fields and dynamics on Lie groups}

In this section, we introduce the notion of linear vector fields and summarize their main properties. 
Let $G$ be a connected Lie group with Lie algebra $\mathfrak{g}$, which we consider as the set of right-invariant vector fields on $G$.

A vector field $\mathcal{X}$ on $G$ is said to be \emph{linear} if its flow $\{\varphi_t\}_{t\in\mathbb{R}}$ forms a one-parameter subgroup of $\mathrm{Aut}(G)$. 
Associated with any linear vector field $\mathcal{X}$ there is a derivation $\mathcal{D}$ of $\mathfrak{g}$ satisfying
\begin{equation*}
(d\varphi_{t})_{e} = \mathrm{e}^{t\mathcal{D}}, 
\qquad \text{for all } t\in\mathbb{R}.
\end{equation*}
In particular,
\begin{equation*}
\varphi_t(\exp Y) = \exp(\mathrm{e}^{t\mathcal{D}}Y),
\qquad \text{for all } t\in\mathbb{R}, \; Y\in\mathfrak{g}.
\end{equation*}

For any connected Lie group $G$, we denote by $T(G)$ the \emph{toral component} of $G$, that is, the maximal compact connected subgroup of $Z(G)_1$. 
It is well known that if $T(G)$ is trivial, then $Z(G)_1$ is simply connected and $T(G/T(G))$ is trivial 
(see \cite[Proposition~3.3]{Mau}).

Let $\{\varphi_t\}_{t\in\mathbb{R}} \subset \mathrm{Aut}(G)$ be a flow. 
By maximality, we have $\varphi_t(T(G)) = T(G)$ for every $t\in\mathbb{R}$. 
However, since the automorphism group of $T(G)$ is discrete, it follows that 
$\varphi_t|_{T(G)} = \mathrm{id}_{T(G)}$. 
In particular,
\[
T(G) \subset \mathrm{fix}(\varphi).
\]

Let $\mathcal{D} = \mathcal{D}_{\EC} + \mathcal{D}_{\HC} + \mathcal{D}_{\NC}$ 
be the Jordan decomposition of $\mathcal{D}$. 
Since $\mathcal{D}$ is a derivation, the components 
$\mathcal{D}_{\EC}, \mathcal{D}_{\HC},$ and $\mathcal{D}_{\NC}$ 
are also derivations of $\mathfrak{g}$ 
(see Theorem~3.2 and Proposition~3.3 of \cite{SM1}). 
Moreover, if $\mathfrak{g}_\lambda$ denotes the eigenspace of the hyperbolic part $\mathcal{D}_{\HC}$, 
then 
\[
[\mathfrak{g}_\lambda, \mathfrak{g}_\mu] \subset \mathfrak{g}_{\lambda+\mu}
\quad \text{if } \lambda+\mu \text{ is an eigenvalue of } \mathcal{D}_{\HC},
\]
and $[\mathfrak{g}_\lambda, \mathfrak{g}_\mu]=0$ otherwise 
(see \cite[Proposition~3.1]{SM1}). 
It follows that
\[
\mathfrak{g}^+ = \bigoplus_{\lambda>0}\mathfrak{g}_\lambda,
\qquad
\mathfrak{g}^- = \bigoplus_{\lambda<0}\mathfrak{g}_\lambda
\]
are nilpotent Lie subalgebras. 
Denoting by $\mathfrak{g}^0 = \ker \mathcal{D}_{\HC}$, we get the decomposition
\[
\mathfrak{g} = \mathfrak{g}^+ \oplus \mathfrak{g}^0 \oplus \mathfrak{g}^-.
\]

Let $\mathcal{X}$ be a linear vector field on a connected Lie group $G$. 
We say that $\mathcal{X}$ is \emph{elliptic}, \emph{hyperbolic}, or \emph{nilpotent}, respectively, 
when its associated derivation $\mathcal{D}$ is elliptic, hyperbolic, or nilpotent. 
The Jordan decomposition of $\mathcal{X}$ is given by
\[
\mathcal{X} = \mathcal{X}_{\EC} + \mathcal{X}_{\HC} + \mathcal{X}_{\NC},
\qquad
\text{where } \mathcal{X}, \mathcal{X}_{\EC}, \mathcal{X}_{\HC}, \mathcal{X}_{\NC}
\text{ commute},
\]
with $\mathcal{X}_{\EC}$ elliptic, $\mathcal{X}_{\HC}$ hyperbolic, and $\mathcal{X}_{\NC}$ nilpotent. 
Following \cite{ASVAPJ}, any linear vector field admits a unique Jordan decomposition. 
If $\{\varphi_t\}_{t\in\mathbb{R}}$ is the flow of $\mathcal{X}$, then
\[
\forall t\in\mathbb{R}, \qquad 
\varphi_t = \varphi_t^{\EC} \circ \varphi_t^{\HC} \circ \varphi_t^{\NC}
\]
is a commutative decomposition of $\varphi_t$, 
where $\{\varphi_t^i\}_{t\in\mathbb{R}}$ is the flow of $\mathcal{X}_i$ for $i = \EC, \HC, \NC$. Moreover, the existence of such decomposition implies that the set of recurrent points for the flow of $\XC$ satisfies (see \cite[Theorem 3.3]{ASVAPJ})
\begin{equation}
    \label{recurrent}
    \RC(\varphi)=\mathrm{fix}(\varphi_t^{\HC})\cap\mathrm{fix}(\varphi_t^{\NC}).
\end{equation}

We define the {\it dynamical subgroups} of $G$ associated with the hyperbolic part of $\XC$ by
\[
G^0 = \fix(\varphi^{\HC}), \quad 
G^+ = \exp(\fg^+), \quad 
G^- = \exp(\fg^-).
\]

The next proposition summarizes the main properties of these dynamical subgroups.
Its proof can be found in \cite[Proposition 2.9]{DS}.

\begin{proposition}
\label{dynamical}
The following properties hold:
\begin{enumerate}
    \item $G^-$ and $G^+$ are connected, simply connected Lie groups; 
    \item $G^0$ normalizes $G^+$ and $G^-$, and consequently $G^{+,0} := G^+G^0$ and $G^{-,0} := G^-G^0$ are subgroups of $G$;
    \item $G^- \cap G^0 = G^+ \cap G^0 = G^{-,0} \cap G^+ = G^{+,0} \cap G^- = \{e\}$;
    \item The dynamical subgroups are closed in $G$;
    \item If $G$ is solvable, then $G^0$ is connected and $G = G^- G^{+,0}$. Moreover, $\fix(\varphi) \subset G^0.$
\end{enumerate}
\end{proposition}

We say that $G$ is {\it decomposable} if $G = G^- G^{+,0}$. 
In particular, if $G$ is decomposable and $x \in G$, there exist unique elements $a \in G^-$, $b \in G^+$ and $c \in G^0$ such that $x = a b c.$

In what follows, we construct a left-invariant Riemannian metric on $G^-$ that will be useful in subsequent sections. 
Assume that $G = G^{-,0}$ with $G^0$ a compact subgroup. 
Let $\langle\cdot, \cdot\rangle$ be an inner product on $\fg^-$ satisfying
\[
\left\langle \rme^{t\DC_{\EC}}X, \rme^{t\DC_{\EC}}Y \right\rangle 
    = \langle X, Y \rangle, 
    \quad \forall X, Y \in \fg^-, \; t \in \R,
\]
which is guaranteed by the $\DC_{\EC}$-invariance of $\fg^-$ and the fact that $\rme^{t\DC_{\EC}}$ is an isometry for all $t \in \R$. 
Define
\[
\llangle X, Y \rrangle := 
    \int_{G^0} \langle \Ad(g)X, \Ad(g)Y \rangle \, d\mu(g),
\]
where $\mu$ denotes the Haar measure on $G^0$. 
It is straightforward to verify that both $\rme^{t\DC_{\EC}}$ and $\Ad(g)$ are isometries of the inner product $\llangle \cdot, \cdot \rrangle$ for all $t \in \R$ and $g \in G^0$. 
Moreover, by equation~(\ref{expanding}), there exist constants $\lambda > 0$ and $c \geq 1$ such that
\[
|\rme^{t\DC}X| \leq c^{-1}\rme^{-t\lambda}|X|, 
    \quad \forall t \geq 0, \; X \in \fg^-.
\]
Let $\varrho$ denote the left-invariant metric on $\fg^-$ induced by $\llangle \cdot, \cdot \rrangle$. 
By construction, the maps $C_g$ (for $g \in G^0$) and $\varphi_t^{\EC}$ are isometries of $\varrho$, and there exists $\lambda > 0$ and $c\geq 1$ such that
\[
\varrho(\varphi_t(g), \varphi_t(h))
    \leq c^{-1}\rme^{-\lambda t}\varrho(g, h),
    \quad \forall g, h \in G^-, \; t \geq 0.
\]

We conclude this section with a technical lemma showing that the only connected Lie groups whose elements have bounded positive powers are necessarily compact. 
Since we did not find proof of this result in the literature, we provide one here for completeness.

\begin{lemma}
\label{bound}
Let $G$ be a connected Lie group. 
If, for every $g \in G$, the set $\{g^n : n \in \N\}$ is bounded, then $G$ is compact.
\end{lemma}

\begin{proof}
By hypothesis and standard properties of the exponential map, each one-parameter subgroup 
$\{\exp(tX) : t \in \R\}$ is bounded for any $X \in \fg$. 
By continuity,
\[
\{\Ad(\exp tX) : t \in \R\} 
    = \{\rme^{t\ad(X)} : t \in \R\}
\]
is bounded in $\Ad(G)$, and hence in $\mathfrak{gl}(\fg)$. 
Examining the Jordan form of $\rme^{t\ad(X)}$ shows that $\ad(X)$ is skew-symmetric. 
In particular, the Cartan–Killing form
\[
\mathcal{K} : \fg \times \fg \to \R,
    \quad (X, Y) \mapsto \mathcal{K}(X, Y) = \tr(\ad(X)\ad(Y)),
\]
is negative semi-definite. 
Indeed, since $\ad(X)$ is skew-symmetric, there exists a basis in which it is block-diagonal with $2 \times 2$ blocks of the form,
\[
A_i = 
\begin{pmatrix}
0 & -\mu_i \\[2pt]
\mu_i & 0
\end{pmatrix}
\quad \implies \quad 
\mathcal{K}(X, X) = \tr(\ad(X)^2) = -\sum_i \mu_i^2 \leq 0.
\]
The above expression also shows that,
\[
\mathcal{K}(X, X) = 0 
\iff \ad(X) \equiv 0 
\iff X \in \fz(\fg).
\]
Therefore, the Cartan–Killing form of the quotient Lie algebra $\fg / \fz(\fg)$ is negative-definite, implying that $G / Z(G)_1$ is a compact semisimple Lie group.

On the other hand, fix a basis $\{X_1, \ldots, X_n\} \subset \fz(\fg)$. 
Since one-parameter subgroups are bounded, there exist compact subsets 
$C_1, \ldots, C_n \subset Z(G)_1$ such that 
$\{\exp(tX_i) : t \in \R\} \subset C_i$. 
Hence, for any $X \in \fz(\fg)$, there exist $\alpha_1, \ldots, \alpha_n \in \R$ such that
\[
\exp X 
    = \exp\!\left(\sum_{i=1}^n \alpha_i X_i\right)
    = \exp(\alpha_1 X_1) \cdots \exp(\alpha_n X_n)
    \subset C_1 C_2 \cdots C_n,
\]
which implies that $Z(G)_1 \subset C_1 C_2 \cdots C_n$ is a compact subgroup. 
Therefore, $G$ is a compact Lie group.
\end{proof}

\section{Dynamical characterization of stable, unstable and central subgroups}

In this section, we show that bounded orbits in $G^0$ necessarily consist of recurrent points of the flow $\varphi$. 
As a direct consequence, we obtain that the dynamical subgroups $G^+$ and $G^-$ correspond, respectively, 
to the unstable and stable manifolds of the identity element of $G$ with respect to the flow $\varphi$.

\subsection{Bounded orbits in $G^0$}
In this section, we prove that every element of the central subgroup $G^0$ whose positive orbit is bounded belongs to the recurrent set. We start with the following lemma.

\begin{lemma}
\label{boundedG^0}
    Let $\varphi$ be a flow of automorphisms on $G$. Then, for any $g\in G^0$, it holds that 
    $$
    \{\varphi_t(g),\, t\geq 0\}\text{ is bounded }
    \iff 
    g\in\fix(\varphi^{\NC}).
    $$
\end{lemma}

\begin{proof}
Since $G^0=\fix(\varphi^{\HC})$, we have for any $g\in G^0$
$$
\{\varphi_t(g),\, t\geq 0\}\text{ bounded }
\iff
\{\varphi^{\NC}_t(g),\, t\geq 0\}\text{ bounded.}
$$
Hence, it suffices to prove that
$$
\{\varphi^{\NC}_t(g),\, t\geq 0\}\text{ bounded }
\implies 
g\in\fix(\varphi^{\NC}).
$$

Let $\varphi^{\Ad}$ denote the flow of automorphisms on $\Ad(G)$ induced by $\varphi^{\NC}$. 
If $x=\Ad(g)$, then by continuity
$$
\{\varphi^{\NC}_t(g),\, t\geq 0\}\text{ bounded in }G
\implies 
\{\varphi^{\Ad}_t(x),\, t\geq 0\}\text{ bounded in }\Ad(G).
$$
By \cite[Proposition 3.1]{ASVAPJ}, $\varphi^{\Ad}_t=\rme^{t\ad(\DC_{\NC})}\big|_{\Ad(G)}$. 
Since the topology of $\Ad(G)$ is finer than that of $\mathfrak{gl}(\fg)$, we get
$$
\{\varphi^{\Ad}_t(x),\, t\geq 0\}\text{ bounded in }\Ad(G)
\implies
\{\rme^{t\ad(\DC_{\NC})}x,\, t\geq 0\}\text{ bounded in }\mathfrak{gl}(\fg).
$$
As $\ad(\DC_{\NC})$ is nilpotent, we obtain
$$
\{\rme^{t\ad(\DC_{\NC})}x,\, t\geq 0\}\text{ bounded }
\implies 
\rme^{t\ad(\DC_{\NC})}x=x,
$$
and therefore the curve
$$
\gamma(t):=g^{-1}\varphi^{\NC}_t(g)
$$
lies in the connected component $Z(G)_1$ of the center of $G$.

If $Z(G)$ is discrete, $\gamma(t)\equiv e$ and hence $g\in\fix(\varphi)$.
If $Z(G)$ is simply connected and $\dim Z(G)>0$, there exists a curve $\beta:\R\to\fz(\fg)$ such that $\gamma(t)=\exp(\beta(t))$.
Moreover,
$$
\gamma(t+s)=\gamma(t)\varphi^{\NC}_t(\gamma(s)), \quad \forall t,s\in\R,
$$
which implies
$$
\beta(t+s)=\beta(t)+\rme^{t\DC_{\NC}}\beta(s), \quad \forall t,s\in\R.
$$
Since $\gamma(t)$ is bounded and the exponential map restricted to $Z(G)_1$ is a diffeomorphism, $\beta(t)$ is also bounded. 
By Lemma~\ref{beta}, $\beta(t)\equiv0$, and hence $\gamma(t)\equiv e$, so $g\in\fix(\varphi)$.

If $Z(G)$ is not simply connected, the toral component $T(G)$ of $G$ is nontrivial. 
As $G/T(G)$ has a simply connected center, applying the previous argument to the curve induced by $\gamma$ in $G/T(G)$ yields $\gamma(t)\in T(G)$ for all $t\in\R$. 
Since $\varphi^{\NC}|_{T(G)}=\id_{T(G)}$, we have $\gamma(t+s)=\gamma(t)\gamma(s)$, showing that $\gamma$ is a one-parameter subgroup of $T(G)$. 
Hence, there exists $X\in\fz(\fg)$ such that $\gamma(t)=\rme^{tX}$. 
Because $T(G)$ is a compact torus, there exists a sequence $t_k\to+\infty$ such that
$$
\rme^{t_kX}\to e,
\quad\text{which implies}\quad
\gamma(t_k)\to e
\quad\text{and}\quad
\varphi^{\NC}_{t_k}(g)\to g.
$$
Thus $g\in\RC(\varphi^{\NC})$, and by~\eqref{recurrent} we conclude that $g\in\fix(\varphi^{\NC})$.
\end{proof}

\begin{remark}
In the previous result, the subgroup $G^0$ is not required to be connected.
\end{remark}

The previous lemma implies the following:

\begin{theorem}
\label{elliptical}
    Let $G$ be a connected Lie group and $\varphi$ a flow of automorphisms. Then, 
    $$\{\varphi_t(g), t\geq 0\}\hspace{.1cm}\mbox{ is bounded for all }g\in G^0\hspace{.5cm}\iff\hspace{.5cm}\varphi|_{G^0}\hspace{.1cm}\mbox{ is elliptic.}$$
\end{theorem}
\begin{proof}
    In fact, by the previous lemma,
$$\{\varphi_t(g), t\geq 0\}\hspace{.1cm}\mbox{ is bounded for all }g\in G^0\hspace{.5cm}\iff\hspace{.5cm}G^0\subset \fix(\varphi^{\mathcal{N}}).$$
In this case, $\varphi^{\mathcal{H}}|_{G^0}=\varphi^{\mathcal{N}}|_{G^0}=\id_{G^0}$, showing that $\varphi|_{G^0}=\varphi^{\mathcal{E}}|_{G^0}$, concluding the proof.
\end{proof}

\subsection{Stable and unstable manifolds}

In this section, we consider that the subgroups $G^+$ and $G^-$ are, in fact, the unstable and stable manifolds of the flow $\varphi$ at the identity element.

\begin{theorem}
\label{stable}
    Let $G$ be a connected Lie group and $\varphi$ a flow of automorphisms on $G$. Then, $G^+$ and $G^-$ are, respectively, the unstable and stable manifolds of $\varphi$ at the identity element, that is,   
$$G^-=\{g\in G, \varphi_t(g)\rightarrow e, t\rightarrow+\infty\}\hspace{.5cm}\mbox{ and }\hspace{.5cm}G^+=\{g\in G, \varphi_t(g)\rightarrow e, t\rightarrow-\infty\}.$$
\end{theorem}

\begin{proof}
Let us consider the first case, since the second one is analogous. Now, the fact that $G^-$ is a nilpotent simply connected Lie group (Proposition \ref{dynamical}), implies that the exponential map $\exp:\fg^-\rightarrow G^-$ is a diffeomorphism. In particular, 
$$X\in\fg^- \hspace{.5cm}\implies \hspace{.5cm}\rme^{t\DC}X\rightarrow 0,\hspace{.5cm} t\rightarrow+\infty\hspace{.5cm}\implies \hspace{.5cm}\varphi_t(\exp X)=\exp(\rme^{t\DC}X)\rightarrow e, \hspace{.5cm}t\rightarrow+\infty.$$
Therefore, 
$$G^-\subset\{g\in G, \varphi_t(g)\rightarrow e, t\rightarrow+\infty\}.$$
On the other hand, let $g\in G$ and assume that $\varphi_t(g)\rightarrow e$ as $t\rightarrow+\infty$. Since $G^{+, 0}_1G^-$ is an open $\varphi$-invariant neighborhood of the origin, we must necessarily have that $g\in G^{+, 0}_1G^-$. Hence, $g$ decomposes uniquely as $g=g_1g_2g_3$ with $g_1\in G^+$, $g_2\in G^0_1$ and $g_3\in G^-$. By the uniqueness of the decomposition and the $\varphi$-invariance of the subgroups $G^+, G^0_1$ and $G^-$, we get that 
$$\varphi_t(g_i)\rightarrow e, \hspace{.5cm}t\rightarrow+\infty, \hspace{.5cm}i=1, 2, 3.$$
However, the fact that $\varphi|_{G^+}$ is expanding in positive-time, gives us that
$$\varphi_t(g_1)\rightarrow e, \hspace{.5cm}t\rightarrow+\infty, \hspace{.5cm}\iff \hspace{.5cm}g_1=e.$$
On the other hand, 
$$\varphi_t(g_2)\rightarrow e, \hspace{.5cm}t\rightarrow+\infty, \hspace{.5cm}\implies \hspace{.5cm}\{\varphi_t(g_2), t\geq 0\}\hspace{.1cm}\mbox{ is bounded }\hspace{.5cm}\implies\hspace{.5cm}g_2\in\fix(\varphi^{\mathcal{N}}).$$
By considering a metric $\varrho$ such that $\varphi^{\mathcal{E}}$ is an isometry, give us that 
$$\varrho(g_2, e)=\varrho(\varphi_t^{\mathcal{E}}(g_2), e)=\varrho(\varphi_t^{\mathcal{E}}(\varphi_t^{\mathcal{N}}(\varphi_t^{\mathcal{H}}(g_2))), e)=\varrho(\varphi_t(g_2), e)\rightarrow 0,$$
implying that $g_2=e$ and hence $g=g_3\in G^-$,  concluding the proof.
\end{proof}

\section{Linear control systems}

A \emph{linear control system} (LCS for short) on $G$ is determined by the family of ordinary differential equations
\begin{flalign*}
	  && \dot{g}(t) = \mathcal{X}(g(t)) + \sum_{i=1}^{m} u_i(s) Y^{i}(g(t)), &&\hspace{-1cm}\left(\Sigma_{G}\right)
	  \end{flalign*}

where $\mathcal{X}$ is a linear vector field, the $Y^{i}$ are right-invariant vector fields, and ${\bf u} = (u_1,\ldots, u_m) \in \mathcal{U}$, with 
$$\UC=\{{\bf u}:\R\rightarrow\R^m; {\bf u} \mbox{ is piecewise constant and } {\bf u}(t)\in\Omega\},$$
for $\Omega\subset\R^m$ a compact, convex subset with $0\in\inner\Omega$, called the \emph{range} of $\Sigma_G$. For each $x \in M$ and ${\bf u} \in \UC$, the system $\Sigma_G$ admits a unique solution 
$t \mapsto \phi(t, x, {\bf u})$ in the sense of Carath\'eodory, satisfying $\phi(0, x, {\bf u}) = x$. The solutions are defined for all $t \in \R$ and satisfy
\[
    \phi(t, gh, {\bf u}) = \phi(t, g, {\bf u})\varphi_t(h), 
    \qquad \forall\, g, h \in G,\, t \in \R,\, {\bf u} \in \UC.
\]

The set of points \emph{reachable from $x$} and the set of points \emph{controllable to $x$} in time $t > 0$ are defined, respectively, as
\[
\mathcal{O}^{+}_{t}(x) := \big\{ y \in M \;\big|\; \text{there exists } {\bf u} \in \mathcal{U} \text{ such that } y = \phi(t, x, {\bf u}) \big\},
\]
\[
\mathcal{O}^{-}_{t}(x) := \big\{ y \in M \;\big|\; \text{there exists } {\bf u} \in \mathcal{U} \text{ such that } x = \phi(t, y, {\bf u}) \big\}.
\]

The \emph{positive} and \emph{negative orbits} of $x$ are, respectively, 
\[
\mathcal{O}^{+}(x) := \bigcup_{t > 0} \mathcal{O}^{+}_{t}(x), 
\qquad
\mathcal{O}^{-}(x) := \bigcup_{t > 0} \mathcal{O}^{-}_{t}(x).
\]
We say that the linear control system $\Sigma_G$ satisfies the \emph{Lie algebra rank condition (abrev. LARC)} if $\fg$ is the smallest $\DC$-invariant subalgebra containing the vectors $\{Y^1, \ldots, Y^m\}$. In particular, if the LARC is satisfied, the sets
\[
\mathcal{O}^{+}_{\leq t}(x) := \bigcup_{0 < \tau \leq t} \mathcal{O}^{+}_{\tau}(x), 
\qquad
\mathcal{O}^{-}_{\leq t}(x) := \bigcup_{0 < \tau \leq t} \mathcal{O}^{-}_{\tau}(x), 
\quad t > 0,
\]
have nonempty interior. The next result states the main properties of the reachable sets (see \cite[Proposition~2]{Jouan}).

\begin{proposition}
\label{properties}
For an LCS, the following statements hold:
\begin{enumerate}
    \item $\OC^{+}_{t_1+t_2}(e) = \OC^+_{t_1}(e)\varphi_{t_1}\big(\OC^+_{t_2}(e)\big)$, for all $t_1, t_2 > 0$;
    \item $\OC^{+}_{t_1}(e) \subset \OC^{+}_{t_2}(e)$, for all $0 < t_1 < t_2$;
    \item $\OC^{+}_{\leq t}(e)=\OC^{+}_{t}(e)$, for all $t>0$; 
    \item $\OC^{+}_{t}(g) = \OC^{+}_{t}(e)\varphi_T(g)$, for all $T > 0$ and $g\in G$.
\end{enumerate}
\end{proposition}

\begin{definition}
\label{def:conjugated}
Let $G$ and $H$ be connected Lie groups. We say that two linear systems $\Sigma_G$ and $\Sigma_H$, respectively on $G$ and $H$, are \emph{conjugated} if there exists a surjective homomorphism $
\psi:G\rightarrow H$ such that
\begin{equation}
\label{solu}
\psi\left(\phi^G(t, g, u)\right)
    =\phi^H(t, \psi(g), u),
    \hspace{1cm}
    \forall\, g\in G,\; t\in\R,\; u\in\UC.
\end{equation}
The map $\psi$ is said to be a \emph{conjugation} between $\Sigma_G$ and $\Sigma_H$.
\end{definition}

From equation (\ref{solu}), it follows that, if $\varphi_t^G$, $\varphi_t^H$ are the flows associated of the drifts of $\psi$-conjugated linear control systems $\Sigma_G$ and $\Sigma_H$, respectively, then 
$$\psi\circ\varphi_t^G=\varphi^H_t\circ\psi, \hspace{.5cm}\forall t\in\R,$$
which by \cite[Lemma 2.3]{DS}, gives us that 
\begin{equation}
\label{conjugation}
\psi(G^+)=H^+, \hspace{.5cm} \psi(G^0)=H^0\hspace{.5cm}\mbox{ and }\hspace{.5cm} \psi(G^-)=H^-.
\end{equation}

\subsection{LCSs with bounded orbits}

In this section, we prove that control sets whose orbits starting from the origin are bounded, imposing rigorous conditions on the topology of the group. We begin with a lemma concerning LCSs with elliptical drift.

\begin{lemma}
\label{elliptical}
    Let $\Sigma_G$ be a linear control system on a connected Lie group satisfying the LARC. If the drift of $\Sigma_G$ is elliptical, then $\mathrm{int}\OC^+(e)$ is a semigroup. In particular, if the orbits in $\OC^+(e)$ are bounded, the system $\Sigma_G$ is controllable, and $G$ is a compact group.
\end{lemma}

\begin{proof}
    Let us denote by $\varphi$ the flow of the drift of $\Sigma_G$ and let us first show that, if $\varphi$ is elliptical, then $\inner\OC^+(e)$ is $\varphi$-invariant.

    Since $\inner\OC^+(e)$ is always $\varphi$-positively invariant, we only have to prove
    $$\varphi_{-t}(\inner\OC^+(e))\subset \inner\OC^+(e), \hspace{.5cm}\forall t<0.$$

    Now, the fact that $\varphi$ is elliptical gives us that, for any $g\in G$, there exists $t_k\rightarrow+\infty$ such that $\varphi_{-t_k}(g)\rightarrow g$. Let then $g\in\inner\OC^+(e)$ and $t<0$. By the previous, there exists $\tau>0$ large enough such that 
    $$-\tau<t \hspace{.5cm}\mbox{ and }\hspace{.5cm}\varphi_{-\tau}(g)\in\inner\OC^+(e).$$ 
    Hence, 
    $$\varphi_{t}(g)=\varphi_{t-\tau}(\varphi_{-\tau}(g))\in\varphi_{t-\tau}(\inner\OC^+(e))\subset\inner\OC^+(e),$$
    where for the last inclusion we used that $t+\tau>0$ and that $\inner\OC^+(e)$ are $\varphi$-positively-invariant.

    Now, if $g, h\in\inner\OC^+(e)$, there exist $t_1>0$ and $u_1\in\UC$ such that $g=\phi(t_1, e, u_1)$. By the $\varphi$-invariance of $\inner\OC^+(e)$ there also exists $t_2>0$ and $u_2\in\UC$ such that $\varphi_{-t_1}(h)=\phi(t_2, e, u)$, implying, by Proposition \ref{properties} that
    $$gh=g\varphi_{t_1}(\varphi_{-t_1}(h_1))=\phi(t_1, e, u_1)\varphi_{t_1}(\phi(t_2, e, u))=\phi(t_1+t_2, e, u)\in\inner\OC_{t_1+t_2}^+(e)\subset\inner\OC^+(e),$$
    showing that $\inner\OC^+(e)$ is a semigroup.

    Repeating the previous construction allows us to show that, for any  $g\in\inner\OC^+(e)$ and any $n\in\N$ there exist $u_n\in\UC$ and $t_n>0$, obtained by concatenation, such that
    $$g^n=\phi(t_n, e, u_n)\in\inner\OC^+(e).$$
    In particular, this allows us to obtain a control $u\in\UC$ such that 
    $$\{g^n, n\in\N\}\subset\{\phi(t, e, u), t\geq 0\}.$$
 Now, if the orbits in $\OC^+(e)$ are bounded, the subset $H=\overline{\{g^n, n\in\N\}}$ is a compact Lie group. Since $H\cap\inner\OC^+(e)\neq \emptyset$ and the set of elements with finite order on a compact Lie group is compact, there exists $h\in H\cap\inner\OC^+(e)$ satisfying $h^n=e$ for some $n\in\N$.
 Hence, 
 $$e=h^n\in\inner\OC^+(e)\hspace{.5cm}\implies\hspace{.5cm}G=\inner\OC^+(e).$$
 The same reasoning for $\inner\OC^-(e)$ allows us to conclude that $G=\inner\OC^-(e)$, showing the controllability of the system. Again, the hypothesis on the boundedness of the orbits implies that $\{g^n, n\in\N\}$ is bounded for any $g\in G$, which by Lemma \ref{bound} implies that $G$ is a compact group, concluding the proof.
\end{proof}

\begin{theorem}
\label{main}
    Let $G$ be a connected Lie group and $\Sigma_G$ a linear control system on $G$ satisfying the LARC. The following conditions are equivalent:
    \begin{enumerate}
        \item $G=G^{-, 0}$ and $G^0$ is a compact subgroup;
        \item $\mathcal{O}^+(e)$ is bounded;
        \item For any $u\in \UC$, the trajectory $\{\phi(t, e, u), t\geq 0\}$ is bounded.
        \end{enumerate}
\end{theorem}

\begin{proof}
$(1\Rightarrow 2)$  Since $G^0$ is compact and acts on $G^-$ by conjugation, there exists a left-invariant metric $\varrho$ on $G^-$ satisfying 
    $$\varrho(C_k(g), C_k(h))=\varrho(g, h), \hspace{.5cm}\forall g, h\in G^-, k\in G^0.$$

    Moreover, by Section 2.2, for this metric, there exist $c\geq 1$, $\lambda>0$ such that 
$$\varrho(\varphi_t(g), \varphi_t(h))\leq c\rme^{-\lambda t}\varrho(g, h), \hspace{.5cm}\forall g, h\in G^-, t\geq0.$$

Now, the fact that the range of the system is bounded implies that $\overline{\OC^+_1(e)}$ is a compact subset of $G$. Hence, there exists $K\subset G^-$ compact, such that $\overline{\OC^+_1(e)}\subset KG^0$. Let $u\in\UC$ and $t>0$ and consider $m\in\N$, $r\in [0, 1)$ such that $t=m+r$. Applying the cocycle property $m$-times gives us that
$$\phi(t, e, u)=\phi_{1, \theta_{m-1+r}u}(e)\varphi_1\left(\phi_{1, \theta_{m-2+r}u}(e)\right)\varphi_2\left(\phi_{1, \theta_{m-3+r}u}(e)\right)\ldots\varphi_{m-1}\left(\phi_{1, \theta_ru}(e)\right)\varphi_m\left(\phi_{r, u}(e)\right).$$
Since $\phi_{1, \theta_{m-j+r}u}(e), \phi_{r, u}(e)\in \OC^+_1(e)$, there exists $a_j\in K$, $b_j\in G^0$ such that 
$$\phi_{1, \theta_{m-j+r}u}(e)=g_{j-1}h_{j-1}, \hspace{.5cm}j=1, \ldots, m\hspace{.5cm}\mbox{ and }\hspace{.5cm}\phi_{r, u}(e)=g_mh_m,$$
and hence
$$\phi(t, e, u)=g_0h_0\varphi_1(g_1h_1)\varphi_2(g_2h_2)\ldots\varphi_m(g_mh_m).$$
By setting $k_0=e$ and for $j=1, \ldots, m$, $k_j:=h_0\varphi_1(h_1)\cdots \varphi_j(h_j)\in G^0$, the fact that $G^0$ normalizes $G^-$ allows us to write $\phi(t, e, u)$ into its $G^{-}$ and $G^0$ components as
$$\phi(t, e, u)=\underbrace{(k_0g_0k_0^{-1})\varphi_1(k_1g_1k_1^{-1})\cdots \varphi_m((k_mg_mk_m^{-1}))}_{:=\phi^-(t, e, u)\in G^-}k_m.$$
Using the metric previously constructed, gives us 
$$\varrho(\phi^{-}(t, e, u), e)\leq\sum_{j=0}^m\varrho(k_j\varphi_j(g_j)k_j^{-1}, e)=\sum_{j=0}^m\varrho(\varphi_j(g_j), e)$$
$$\leq c^{-1}\sum_{j=0}^m\rme^{-j\lambda}\varrho(g_j, e)\leq c^{-1}\mathrm{diam}(K)\sum_{j=0}^m(\rme^{-\lambda})^j\leq c^{-1}\cdot\mathrm{diam}(K)\sum_{j=0}^{+\infty}(\rme^{-\lambda})^j=\frac{c^{-1}\cdot\mathrm{diam}(K)}{1-\rme^{-\lambda}}:=R,$$
showing that $\phi^-(t, e, u)\in \overline{B_{G^-}(e, R)}$, and hence,  $\phi(t, e, u)\in \overline{B_{G^-}(e, R)} G^0.$ By the arbitrariness of $u\in\UC$ and $t>0$, we conclude that
$$\mathcal{O}^+(e)\subset \overline{B_{G^-}(e, R)}G^0,$$
showing that $\mathcal{O}^+(e)$ is  bounded.

\bigskip

$(2\Rightarrow 3)$ direct.

\bigskip

$(3\Rightarrow 1)$ Let us assume now that the orbits of the linear control system starting at the origin are bounded. Since the identity element belongs to the closure of $\mathrm{int}\OC^+(e)$, it holds that
$$V=G^+G_1^{-, 0}\cap\mathrm
{int}\OC^+(e),$$
is a nonempty open set of $G$. In particular, for any $g\in V$, there exists $u\in\UC$, $\tau>0$ such that $g=\phi(\tau, e, u)$. By concatenating $u$ with the zero control, we obtain that  
$$\{\varphi_t(g), t\geq 0\}\subset\phi(t, e, u'), t\geq 0\},$$
showing that $\{\varphi_t(g), t\geq 0\}$
is bounded. As in the proof of Theorem \ref{stable} we conclude that $g=g_1g_2$ with $g_1\in G^-$ and $g_2\in G^0$, with  $\{\varphi_t(g_2), t\geq 0\}$ is bounded. By Lemma \ref{boundedG^0} we conclude that $g_2\in\RC(\varphi)$ implying that 
$$ V\subset G^-\RC(\varphi).$$
Therefore, $G^-\RC(\varphi)$ has nonempty interior in $G$ and hence $G=G^-\RC(\varphi)$, or equivalently, 
$$G=G^{-, 0}\hspace{.1cm}\mbox{ and }\varphi|_{G^0}\hspace{.1cm}\mbox{ is elliptcal},$$
where the last assertion follows from Theorem \ref{elliptical}.

Since $G=G^{-, 0}$, the subgroup $G^-$ is a normal subgroup and $G/G^-\simeq G^0$. Moreover, the fact that $G^-$ is $\varphi$-invariant, allows us to induce a LCS on $G/G^-$ whose drift is, by the previous, elliptical. By Lemma \ref{elliptical} we conclude that $G^0$ is a compact subgroup, thus concluding the proof.
\end{proof}

In what follows, we formally define the concept of control sets (for more on control sets, the reader can consult \cite{FCWK}).

\begin{definition}
\label{controlset}
A subset $D \subset G$ is called a \emph{control set} of the linear control system $\Sigma_G$ if it satisfies the following properties:

\begin{itemize}
    \item[(i)] \emph{(Weak invariance)} For every $x \in M$, there exists ${\bf u} \in \mathcal{U}$ such that 
    \(\phi(\R^+, x, {\bf u}) \subset D\);
    \item[(ii)] \emph{(Approximate controllability)} $D \subset \overline{\mathcal{O}^+(x)}$ for every $x \in D$;
    \item[(iii)] \emph{(Maximality)} $D$ is maximal with respect to properties (i) and (ii).
\end{itemize}
\end{definition}

In particular, when the whole state space $G$ is a control set, the system $\Sigma_G$ is said to be \emph{controllable}. 
Control sets are pairwise disjoint and contain several critical dynamical properties of the system, such as fixed and recurrent points, as well as periodic and bounded orbits. 
Moreover, exact controllability holds in the interior of a control set; that is, points in the interior can be steered into one another by a solution of the system in positive time.

As a direct consequence of Theorem~\ref{main}, we get the following result, which ensures the existence of control sets.

\begin{corollary}
    If $\Sigma_G$ is an LCS on a connected Lie group satisfying the LARC, then  
    $$\OC^+(e)\hspace{.1cm} \mbox{ is bounded}\hspace{.5cm}\iff\hspace{.5cm}\overline{\OC^+(e)} \hspace{.1cm}\mbox{ is a compact control set}.$$
    In particular, if $\Sigma_G$ is controllable, then $G$ is a compact group.
\end{corollary}

\begin{proof}
    We only have to show that $\overline{\OC^+(e)}$ is a control set when $\OC^+(e)$ is bounded. Under such an assumption, Theorem \ref{main} implies that $G=G^{-, 0}$ with $G^0$ compact. Since the induced system on $G/G^-\simeq G^0$ is controllable, for any $x\in G$, there exists $t>0$, $u\in\UC$ and $g\in G^-$ such that $\phi(t, x, u)=h.$ Concatenation with the zero control, gives us that
    $$\forall \tau>0, \hspace{.5cm}\OC^+(x)\ni \phi(\tau+t, x, u)=\varphi_{\tau}(h)\rightarrow e\hspace{.5cm}\implies \hspace{.5cm}e\in\overline{\OC^+(x)}, \hspace{.5cm}\forall x\in G.$$ 
The previous imply that 
$$\forall x\in \overline{\OC^+(e)}, \hspace{.5cm}\overline{\OC^+(x)}=\overline{\OC^+(e)},$$
implying that $\overline{\OC^+(e)}$ satisfies properties 1. and 2. in Definition \ref{controlset}. Moreover, the previous forces also the maximality of $\overline{\OC^+(e)}$, concluding the proof.   
\end{proof}

\section{Stabilization of linear control systems on Lie groups}

The result of the previous section allows us to extend and characterize the classical concepts of inner stability and BIBO stability for linear control systems on Lie groups.

\subsection{Internal stability}

In this section we extend the concept of internal stability presented in ??? to the context of linear control systems on Lie groups and show that this concept depends heavily on the dynamical behavior of the drift of the system.

\begin{definition}
    Let $\Sigma_G$ be a linear control system on a connected Lie group $G$ and denote by $\{\varphi_t\}_{t\in\R}$ the flow of the drift of $\Sigma_G$. We say that:
\begin{itemize}
    \item[(i)] The linear control system $\Sigma_G$ is said to be \emph{internally stable} if 
    $$\{\varphi_t(g), t\geq 0\}\mbox{ is bounded } \forall g\in G;$$
     \item[(ii)] The linear control system $\Sigma_G$ is said to be \emph{internally asymptotically stable} if 
    $$\forall g\in G, \hspace{.5cm} \varphi_t(g)\rightarrow e, \hspace{.5cm}t\rightarrow+\infty;$$
    \item[(iii)] The linear control system $\Sigma_G$ is said to be \emph{internally unstable}   
\end{itemize}
\end{definition}

As in the Euclidean case, internal stability has nothing to do with any other part of $\Sigma_G$ other than its drift. The next result characterizes internal stability in terms of the dynamical subgroups of the drift.

\begin{theorem}
    \label{internal}
    For a linear control system $\Sigma_G$, the following statments holds:
    \begin{itemize}
        \item[(i)] $\Sigma_G$ is internally unstable if $G^+$ is not trivial;
        \item[(ii)] $\Sigma_G$ is internally asymptotically stable if $G=G^-$;
        \item[(iii)] $\Sigma_G$ is internally stable if $G=G^{-, 0}$ and $\varphi|_{G^0}$ is elliptical;
        \item[(iv)] $\Sigma_G$ is internally unstable if $\varphi|_{G^0}$ is not elliptical.
    \end{itemize}
\end{theorem}

\begin{proof}
    The items (i) and (ii) follow from Theorem \ref{stable}, and the items (iii) and (iv) follow from Theorem \ref{elliptical}.  
\end{proof}

\subsection{BIBO stability}

We now consider a concept of stability that takes into account the control part of a linear control system.

\begin{definition}
    Let $G, H$ be connected Lie groups and $F:G\rightarrow H$ a group homomorphism. We say that a linear control system $\Sigma_G$ on $G$ is bounded input-bounded output stable  (abrev. BIBO stable) relative to $F$, if 
    $$\forall u\in\UC, \hspace{.5cm}\{F(\phi(t, e, u)), t\geq 0\}\hspace{.5cm}\mbox{ is bounded}.$$
\end{definition}


 \begin{theorem}
     Let $G, H$ be connected Lie groups and $F:G\rightarrow H$ be a homomorphism. Let $\Sigma_G$ be a linear control system on $G$ and assume that $\ker F$ is invariant by the flow of the drift. Then, $\Sigma_G$ is BIBO stable relative to $F$ if and only if 
     $$G^+\subset \ker F\hspace{.5cm}\mbox{ and }\hspace{.5cm} F(G^0)\hspace{.3cm}\mbox{ is a compact subgroup of the image }F(G).$$
 \end{theorem}

\begin{proof}
    Since $\mathrm{ker}F$ is $\varphi$-invariant, the system $\Sigma_G$ induces a linear control system $\Sigma_{F(G)}$ on the image $F(G)\simeq G/\ker F$. Hence, $\Sigma_G$ is BIBO stable if and only if $\{\phi^{F(G)}(t, x, u), t\geq 0\}$ is bounded for any $u\in\UC$, where $\phi^{F(G)}$ is the solution of the induced system. However, by Theorem \ref{main}, $\{\phi^{F(G)}(t, x, u), t\geq 0\}$ is bounded for any $u\in\UC$ if and only if $F(G)=F(G)^{-, 0}$ and $F(G)^0$ is a compact subgroup. Since, by relation (\ref{conjugation}), it holds that 
    $$F(G^-)=F(G)^-, \hspace{.5cm}F(G^+)=F(G)^+ \hspace{.5cm}\mbox{ and } \hspace{.5cm}F(G^0)=F(G)^0,$$
    the result follows.
\end{proof}

\begin{remark}
The invariance of $\ker F$ can be verified in many cases, such as:
\begin{enumerate}
    \item If $G$ is a semisimple Lie group, any flow of automorphisms is inner, hence, any normal subgroup is invariant;
    \item If $\ker F$ is such that $\mathrm{Lie}(\ker F)$ is a term of the derivative series or of the central descending series, the system is BIBO stable.
\end{enumerate}
\end{remark}

\end{document}